\documentclass[13pt]{article}  
\usepackage{amssymb}              
\usepackage{amsthm}
\usepackage{amsmath}
\usepackage{eucal}
\usepackage{verbatim}
\usepackage[dvips]{graphicx}
\usepackage{multirow}
\usepackage{fancyhdr}
\usepackage{color}
\usepackage{enumerate}
\usepackage{calrsfs}
\usepackage{fullpage}
\usepackage{eufrak}

\newtheorem{theorem}{Theorem}
\newtheorem{lemma}{Lemma}
\newtheorem{definition}{Definition}

\makeatletter
\newcommand{\leqnomode}{\tagsleft@true}
\newcommand{\reqnomode}{\tagsleft@false}
\makeatother

\def\({\begin{eqnarray}}
\def\){\end{eqnarray}}
\def\[{\begin{eqnarray*}}
\def\]{\end{eqnarray*}}
\def\part#1#2{\frac{\partial #1}{\partial #2}}

\def\R{\mathbb{R}}
\def\N{\mathbb{N}}
\def\d{\mathrm{d}}
\def\tot#1#2{\frac{\d #1}{\d #2}}
\def\eps{\varepsilon}

\def\Norm#1{\left\| #1 \right\|}
\def\laplace{\Delta}

\def\wtx{\widetilde x}
\def\wx#1#2{{\wtx_{#1}}^{\; #2}}
\def\wxji{\wx{j}{i}}
\def\wxij{\wx{i}{j}}
\def\ta#1#2{\tau_{{#1}{#2}}}
\def\taij{\ta{i}{j}}

\def\c{{\mathfrak{c}}}
\def\s{{\mathfrak{s}}}

\def\grad{\nabla}
\def\S{\mathbb{S}}

\def\rev#1{{#1}}

\begin{document}

\title{Well posedness and asymptotic consensus in the Hegselmann-Krause model with finite speed of information propagation}   
\author{Jan Haskovec\footnote{Computer, Electrical and Mathematical Sciences \& Engineering, King Abdullah University of Science and Technology, 23955 Thuwal, KSA.
jan.haskovec@kaust.edu.sa}}

\date{}

\maketitle

\begin{abstract}
We consider a variant of the Hegselmann-Krause model of consensus formation where
information between agents propagates with a finite speed $\c$.
This leads to a system of ordinary differential equations (ODE) with state-dependent delay.
Observing that the classical well-posedness theory for ODE systems
does not apply, we provide a proof of global existence and uniqueness of solutions of the model.
We prove that asymptotic consensus is always reached in the spatially one-dimensional setting of the model,
as long as agents travel slower than $\c$.
We also provide sufficient conditions for asymptotic consensus in the spatially multi-dimensional setting.
\end{abstract}
\vspace{2mm}

\vspace{2mm}

\textbf{2010 MR Subject Classification}: 34K20, 34K60, 82C22, 92D50. 
\vspace{2mm}

\section{Introduction}\label{sec:Intro}
Individual-based models of collective behavior attracted the interest of researchers
in several scientific disciplines. A particularly interesting aspect of the dynamics of multi-agent systems
is the emergence of global self-organizing patterns, while individual agents
typically interact only locally. This is observed in various types of systems -
physical (e.g.,  spontaneous magnetization and crystal growth in classical physics),
biological (e.g., flocking and swarming, \cite{Camazine, Vicsek}) or socio-economical \cite{Krugman, Naldi, Castellano}.
The field of collective (swarm) intelligence also found many applications in engineering and robotics
\cite{Hamman, Valentini, Jadbabaie}.

An important area in the study of collective behavior is opinion dynamics \cite{Xu}.
In this paper, we focus on a widely used model referred to as the
Hegselmann-Krause consensus model \cite{HK}.
It describes the evolution of $N\in\N$ agents who adapt their opinions to the ones of their close neighbors.
Agent $i$'s opinion is represented by the quantity $x_i\in\R^d$, $d\in\N$, which is a function of time $t\geq 0$
and evolves according to the following dynamics
\( \label{eq:HK}
   \dot x_i(t) = \frac{1}{N-1} \sum_{j=1}^N \psi(|x_j(t) - x_i(t)|) (x_j(t) - x_i(t)), \qquad i=1,\ldots,N.
\)
The nonnegative function $\psi:[0,\infty)\to [0,\infty)$ is the so-called \emph{influence function} and measures
how strongly each agent is influenced by others depending on their distance.
The phenomenon of consensus finding in the context of \eqref{eq:HK} refers
to the (asymptotic) emergence of one or more \emph{opinion clusters} formed
by agents with (almost) identical opinions. \emph{Global consensus}
is the state where all agents have the same opinion, i.e.,
$x_i=x_j$ for all $i,j \in\{1,\dots,N\}$.
Various aspects of the consensus behavior of various modifications of \eqref{eq:HK}
have been studied in, e.g., \cite{Bha, Blondel, Canuto, Carro, Mohajer, Moreau, MT, JM, Wang, Wedin}.

In certain applications it may be relevant to account for the finite speed of information propagation.
For instance, in radio communication between satellites on the orbit or in outer space,
where the distances are not negligible with respect to the speed of light. However, the
finiteness of the speed of light may be a relevant factor even in terrestrial conditions,
for example for high-speed stock traders. This is best illustrated by the fact that a transatlantic fibre-optic cable has been laid
recently to speed up the connection between the stock trading firms and banks in the City of London
and New York by six milliseconds from the previous 65 milliseconds \cite{telegraph};
we note though that light in optical fiber cables travels about 30\% slower than in vacuum.
These possible applications motivate us to introduce a modification of the Hegselmann-Krause model \eqref{eq:HK}
where information propagates with a finite speed $\c>0$, called \emph{speed of light} in the sequel.
This means that agent located in $x_i=x_i(t)$ at time $t>0$ observes the
position of the agent $x_j$ at time $t-\taij$, where $\taij$ solves
\(  \label{eq:tau}
   \c \taij(t) = |x_i(t) - x_j(t-\taij(t))|,
\)
i.e., $\taij(t)$ is the time that information (light) needs to travel from location $x_j(t-\taij(t))$
to location $x_i(t)$, \rev{for all $i, j\in\{1,\dots,N\}$.}
In general it is neither guaranteed that a solution of \eqref{eq:tau}
exists nor that it is unique. This issue is of course related
to the possibility of agents traveling faster than the speed of light $\c$.
We shall formulate sufficient conditions for the well-posedness
of the model in the course of our analysis.
At the time being, let us assume that \eqref{eq:tau} is uniquely solvable with solution $\taij\geq 0$
and introduce the following notation
\[
   \wxji := x_j(t - \taij(t)).
\]
We also introduce the formal notation $\ta{i}{i}:=0$ and $\wx{i}{i}:=x_i(t)$
and, if no danger of confusion, we shall usually drop the explicit time dependence,
writing just $x_i$ for $x_i(t)$.
With this notation, the system that we study in this paper is written as
\(  \label{eq:0}
   \dot x_i = \frac{1}{N-1} \sum_{j=1}^N \psi(|\wxji - x_i|) \left( \wxji - x_i \right), \qquad i=1,\ldots,N.
\)
We shall also frequently use the shorthand notation for the \emph{influence rates}
\[  
   \widetilde\psi_{ij} := \psi\left( \left| \wxji - x_i \right| \right).
\]
The system \eqref{eq:0} is equipped with the initial datum
\(  \label{IC:0}
    x_i(t) = x_i^0(t)\qquad\mbox{for } i=1,\dots,N,\quad t\leq 0,
\)
where $x_i^0=x_i^0(t)$ are Lipschitz continuous paths on $(-\infty,0]$.
Depending on the speed of light $\c$ and the speed of the individual agents,
the initial datum is only relevant on a bounded time interval.
We shall make this dependence explicit later.

To our best knowledge, the Hegselmann-Krause model with time delay
has only been studied in \cite{Lu} and \cite{Choi}.
In \cite{Lu} the authors consider the case where the delays $\tau_{ij}$
are constant and a priori given. Also the influence rates
(in our notation $ \widetilde\psi_{ij}$) are fixed.
The authors prove that if the influence rates correspond
to a strongly connected graph, then the system reaches asymptotic consensus,
whatever the values of the delay $\tau_{ij}$, both in the case of linear and nonlinear coupling.
Their results are based on the construction of a suitable Lyapunov-Krasovskii functional.
The paper \cite{Choi} considers the case of variable delay $\tau=\tau(t)$, shared by all agents.
The influence rates depend on the agent distance. Based on convexity properties
and a Lyapunov functional, the authors prove asymptotic consensus in the system
under the assumption that the delay is uniformly small, measured by the decay
of the influence function $\psi$.

The main novelty introduced in this work is the fact that the delay
in \eqref{eq:0} depends on the configuration of the system
in a nontrivial way through \eqref{eq:tau}.
This \emph{state dependent delay} \cite{Smith} poses new analytical challenges:
The standard well-posedness theory for ODE systems
does not apply to \eqref{eq:tau}--\eqref{IC:0}.
Moreover, the known Lyapunov functional-type approaches
fail and new methods need to be developed to study
the asymptotic consensus behavior.
This paper aims at addressing these issues.

\section{Overview of main results}\label{sec:mainRes}
The benefits of this paper are threefold: First, observing that the classical theorems
of Peano and Picard-Lindel\"of do not apply to the system \eqref{eq:tau}--\eqref{IC:0},
we provide a proof of existence and uniqueness of its solutions (Section \ref{sec:twoAgents}).
Second, in Section \ref{sec:1D} we prove that asymptotic consensus is always reached
in the spatially one-dimensional setting of the model.
Finally, in Section \ref{sec:multiD} we provide sufficient conditions for asymptotic consensus
in the spatially multi-dimensional setting.

Let us observe that, in general, \eqref{eq:tau} can only be uniquely solvable if the agents
move with speeds strictly less than the speed of light $\c$.
This motivates us, for $0 < \s < \c$ and any interval $\mathcal{I}\subseteq\R$, to introduce the space
\[  
   C_\s(\mathcal{I},\R^d) := \left\{ \varphi: \mathcal{I} \to \R^d; \varphi \mbox{ is uniformly Lipschitz continuous on } \mathcal{I}
     \mbox{ with Lipschitz constant } \s\right\}.
\]
Then, it is natural to require that the initial datum $x^0$ in \eqref{IC:0} is an element of $C_\s(\mathcal{I},\R^d)^N$.
Moreover, to guarantee that agents observe the subluminal speed limit during the evolution driven by \eqref{eq:0},
we pose the assumption
\(   \label{cond:s}
    \s := \sup_{r>0} \psi(r)r  < \c.
\)
\rev{This assumption may seem restrictive at first glance, however,
for a fixed initial datum the boundedness of $\psi(r)r$ is in fact only required
on bounded $r$-intervals of length determined by the radius of the support
of the initial datum. Indeed, as we shall prove below, the system \eqref{eq:tau}--\eqref{IC:0}
is non-expansive, i.e., all particle trajectories are uniformly contained within a
compact set. 
In other words, once the initial datum is fixed, the asymptotic properties
of $\psi(r)$ as $r\to\infty$ are irrelevant.
On the other hand, let us point out that the assumption \eqref{cond:s}
is necessary to prove the well-posedness of the system \eqref{eq:tau}--\eqref{eq:0}
for all admissible initial data \eqref{IC:0}. Indeed, if we had $\psi(r)r \geq \c$ on a set of positive measure,
then \eqref{eq:0} would force (some) particles to move at the speed of light or faster,
destroying the unique solvability of \eqref{eq:tau}.
Consequently, \eqref{cond:s} cannot be relaxed.
}

Moreover, we assume that $\psi(r)>0$ for all $r>0$ and
that $\psi$ is uniformly Lipschitz continuous on $[0,\infty)$
with Lipschitz constant $L_\psi$.
Clearly, this assumption combined with \eqref{cond:s} implies
that $\psi$ is globally bounded.
Moreover, we have $\lim_{r\to 0} \psi(r)r = 0$, so that the global consensus, i.e.,
$x_i\equiv x_j$ for all $i,j\in \{1,\dots,N\}$, is an equilibrium for \eqref{eq:tau}--\eqref{eq:0}.

We then have the following result about the well-posedness of the system  \eqref{eq:tau}--\eqref{IC:0}.

\begin{theorem}\label{theorem:WP}
Let the above assumptions on the influence function $\psi$ be verified and
let the initial datum $x^0\in C_\s(\mathcal{I},\R^d)^N$.
Then the system \eqref{eq:tau}--\eqref{IC:0} possesses unique global solutions.
\end{theorem}

We shall prove this theorem for a simplified scalar equation in Section \ref{sec:twoAgents},
noting that the proof for the original system \eqref{eq:tau}--\eqref{IC:0} is principally the same.

Our second result is about asymptotic global consensus finding of the system.
For this sake, let us introduce for $t\in\R$ the diameter of the particle group,
\(   \label{def:diam}
   d_x(t) := \max_{i,j\in\{1,\ldots,N\}} |x_i(t) - x_j(t)|.
\)
Clearly, global consensus is equivalent to $d_x=0$.
As we shall prove in Section \ref{sec:1D}, the above assumptions on $\psi$
are sufficient for asymptotic consensus finding, i.e.,
$\lim_{t\to\infty} d_x(t) = 0$, in the spatially one-dimensional case.
Note that we do not exclude the possibility that global consensus may be reached
also in finite time.
In the multidimensional case the situation is more complicated
and we were not able to prove a similar "unconditional" global consensus finding
result as in 1D. 
However, in Section \ref{sec:multiD} (Theorem \ref{thm:shrinkage}) we provide
a sufficient condition for asymptotic consensus, formulated in terms of the speeds $\s$, $\c$
and the upper and lower bound of the influence function $\psi$ on a certain compact interval.
We strongly conjecture that asymptotic global consensus is always found even in the
multidimensional case, under the same assumptions as those of Theorem \ref{theorem:WP}.
We leave the proof of this conjecture for a future work.
We also note that, in general, \eqref{eq:0} does not conserve the mean value
$\frac{1}{N}\sum_{i=1}^N x_i$. Consequently, the (asymptotic) consensus vector cannot be
inferred from the initial datum in a straightforward way and can be seen as an emergent property of the system.

\subsection{Initial datum and solvability of \eqref{eq:tau}}
\begin{lemma} \label{lem:tau}
Let $x\in C_\s((-\infty,t]; \R^d)$ for some $t\in\R$ and $\s<\c$.
Then the equation
\(  \label{tauxz}
   \c \tau = |z-x(t-\tau)|
\)
is uniquely solvable in $\tau\geq 0$ for each $z\in\R^d$ and the solution $\tau$ satisfies
\(  \label{bound:tau}
   \frac{|z-x(t)|}{\c+\s} \leq \tau \leq \frac{|z-x(t)|}{\c-\s}.
\)
\end{lemma}

\begin{proof}
For all \rev{$s\geq 0$} we readily have
\[
   |z-x(t-s)| \leq |z-x(t)| + |x(t)-x(t-s)| \leq  |z-x(t)| + \s s.
\]
Therefore, the function $u(s) := |z-x(t-s)| - \c s$ satisfies
\[
   u(0) = |z-x(t)| \geq 0, \qquad u(s) \leq |z-x(t)| + (\s-\c) s.
\]
Due to the continuity of $u$ and the fact that $\s-\c<0$, there exists some \rev{$0\leq \tau\leq \frac{|z-x(t)|}{\c-\s}$} such that $u(\tau)=0$.
Moreover, we have
\[
   |z-x(t)| \leq |z-x(t-\tau)| + |x(t-\tau)-x(t)| \leq (\c+\s) \tau,
\]
and the lower bound in \eqref{bound:tau} follows.

Finally, let us assume, for contradiction, that there exist nonnegative $\tau_1\neq \tau_2$,
both being solutions of \eqref{tauxz}. Then with the triangle inequality we have
\[
   \c |\tau_1-\tau_2| = \bigl| |z-x(\tau_1)| - |z-x(\tau_2)|\bigr| \leq |x(\tau_1)-x(\tau_2)| \leq \s |\tau_1-\tau_2|,
\]
which is a contradiction to $\s<\c$. Consequently, the solution $\tau$ is unique.
\end{proof}

Applying the result of Lemma \ref{lem:tau} for $t=0$, we see that the values of the initial datum
$x^0\in C_\s((-\infty,0]; \R^d)$ in \eqref{IC:0} are relevant at most on the time interval $\left[-\frac{d_x(0)}{\c-\s},0\right]$.

\section{Two agents}\label{sec:twoAgents}
We consider the system \eqref{eq:tau}--\eqref{IC:0} with two agents only, $N=2$,
restricted to the scalar setting $d=1$.
It reduces to a single equation
if we prescribe a symmetric initial datum, say
\[  
   x_1(t) = x^0(t), \qquad x_2(t) = -x^0(t),\qquad \mbox{for } t\leq 0.
\]
Then, we obviously have \rev{$\tau_{12} \equiv \tau_{21}$} and $x_1 \equiv - x_2$,
so that for $x:=x_1$,
\( \label{eq:x1d}
   \dot x = - \psi(|x+\wtx|) (x + \wtx),
\)
where $\wtx := x(t-\tau(t))$ and $\tau=\tau(t)$
solves the equation 
\(  \label{eq:tau1d}
   \c \tau(t) = |x(t) + x(t-\tau(t))|.
\)
By Lemma \ref{lem:tau} we prescribe the initial datum $x^0\in C_\s([-S^0,0];\R)$
for some $\s<\c$, where here and in the sequel we denote $S^0 := \frac{2|x^0(0)|}{\c-\s}$.
We may assume that $|x^0(0)| >0$, since for $x^0(0)=0$
\rev{
we have $S^0 = 0$ and $\tau(0)=0$, so that the unique solution of \eqref{eq:x1d}--\eqref{eq:tau1d}
is the equilibrium $x\equiv 0$.}

Let us observe that, despite the assumption about the Lipschitz continuity
and global boundedness of the response function $\psi$,
the standard Peano or Picard-Lindel\"{o}f theorems on existence/uniqueness
of solutions do not apply to the system \eqref{eq:x1d}--\eqref{eq:tau1d}.
Therefore, we shall demonstrate how unique solutions
of \eqref{eq:x1d}--\eqref{eq:tau1d} can be constructed
using the Banach fixed point theorem in the spirit of Picard-Lindel\"{o}f.
Let us fix the initial datum $x^0\in C_\s([-S^0,0],\R)$ and
for some $T>0$ to be specified later, define the space
\[
   \mathbb{X}_T := \left\{ \varphi\in C_\s([-S^0, T];\R);\, \varphi|_{[-S^0, 0]} \equiv x^0| _{[-S^0, 0]} \right\}.
\]
I.e., $\mathbb{X}_T$ is the space of Lipschitz continuous functions on the interval $[-S^0, T]$
with Lipschitz constant $\s$, which coincide with the initial datum $x^0$ on the interval $[-S^0, 0]$.
We equip the space $\mathbb{X}_T$ with the topology of uniform convergence,
i.e., with the $L^\infty$-norm on $[0,T]$.
Since the space of Lipschitz continuous functions on a compact set with Lipschitz constant $\s$
is closed with respect to the topology of uniform convergence, $\mathbb{X}_T$ is a Banach space. 

Let us define the Picard operator $\Gamma: \mathbb{X}_T \to \mathbb{X}_T$,
\( \label{Picard}
  (\Gamma\varphi)(t) :=  x^0(0) + \int_0^t &\psi(|\varphi(s)+\varphi(s-\tau(s))|) (\varphi(s)+\varphi(s-\tau(s))) \d s \qquad \mbox{for }t\in [0,T],
\)
and we set $\Gamma\varphi$ equal to $x^0$ on $[-S^0,0]$.
The function $\tau=\tau(s)$ is the unique solution of the equation
\( \label{Picard:tau}
   \c\tau(s) = \left| \varphi(s) + \varphi(s-\tau(s)) \right|.
\)
Let us note that, by Lemma \ref{lem:tau}, the above equation is indeed uniquely solvable for each $s\in [0,T]$.
Moreover, we have $s-\tau(s) \geq - S^0$ for all $s\in [0,T]$.

\begin{lemma}\label{lem:locEx}
Let the influence function $\psi$ satisfy \eqref{cond:s} with $\s$.
Then the operator $\Gamma$ defined by \eqref{Picard} 
maps the space $\mathbb{X}_T$ into itself and,
for sufficiently small $T>0$, it is a contraction on $\mathbb{X}_T$
with respect to the $L^\infty(0,T)$-norm.
\end{lemma}

\begin{proof}
Let us pick some $\varphi\in\mathbb{X}_T$, then \eqref{cond:s} immediately implies that
$\Gamma\varphi$ is Lipschitz continuous with constant $\s$ on $[0,T]$.
Since $\Gamma\varphi$ is by definition equal to $x^0$ on $[-S^0,0]$,
it is Lipschitz continuous on the whole interval $[-S^0,T]$ with the same constant.
Therefore, $\Gamma$ maps $\mathbb{X}_T$ into itself.

To prove the contractivity of $\Gamma$, let us pick $\varphi_1, \varphi_2 \in \mathbb{X}_T$
and calculate
\[
   \left|\Gamma\varphi_1(t) - \Gamma\varphi_2(t)\right| &\leq& \int_0^t \bigl| \psi(|\varphi_1(s)+\varphi_1(s-\tau_1(s))|)(\varphi_1(s)+\varphi_1(s-\tau_1(s))) \bigr. \\
      &&\qquad \bigl. -  \psi(|\varphi_2(s)+\varphi_2(s-\tau_2(s))|)(\varphi_2(s)+\varphi_2(s-\tau_2(s))) \bigr| \d s,
\]
where $\tau_1$ and, resp., $\tau_2$ are solutions of \eqref{Picard:tau} with $\varphi_1$, resp., $\varphi_2$.
Using the estimate
\[
   \bigl| \psi(\sigma)\sigma - \psi(\lambda)\lambda \bigr| \leq \left( \Norm{\psi}_{L^\infty(0,\infty)} + \min\{\sigma,\lambda\} L_\psi \right) |\sigma-\lambda|,
\]
for any $\lambda$, $\sigma>0$, where $L_\psi$ is the Lipschitz constant of $\psi$, we have
\[
   \left|\Gamma\varphi_1(t) - \Gamma\varphi_2(t)\right| &\leq&
       \left( \Norm{\psi}_{L^\infty(0,\infty)} + 2(|x^0(0)| + \s S^0 + \s T) L_\psi \right) \\
       && \qquad\times
       \int_0^t \bigl| |\varphi_1(s)+\varphi_1(s-\tau_1(s))| - |\varphi_2(s)+\varphi_2(s-\tau_2(s))| \bigr| \d s,
\]
where we used the bound
\[
   |\varphi_1(s)| \leq |x^0(0)| + \s(T+S^0) \qquad\mbox{for all } s\in [-S^0,T],
\]
implied by the Lipschitz continuity of $\varphi_1$.
We further calculate
\[
   \bigl| |\varphi_1(s)+\varphi_1(s-\tau_1(s))| - |\varphi_2(s)+\varphi_2(s-\tau_2(s))| \bigr| &\leq&
       \Norm{\varphi_1 - \varphi_2}_{L^\infty(0,T)} + |\varphi_1(s-\tau_1(s)) - \varphi_2(s-\tau_2(s))|  \\
       &\leq& 
       2 \Norm{\varphi_1 - \varphi_2}_{L^\infty(0,T)} + |\varphi_2(s-\tau_1(s)) - \varphi_2(s-\tau_2(s))| \\
       &\leq& 
       2 \Norm{\varphi_1 - \varphi_2}_{L^\infty(0,T)} + \s |\tau_1(s) - \tau_2(s)|,
\]
where we used the Lipschitz continuity of $\varphi_2$ for the last inequality.
Now, with \eqref{Picard:tau} we have
\[
   \s |\tau_1(s) - \tau_2(s)| = \s\c^{-1} \bigl| |\varphi_1(s)+\varphi_1(s-\tau_1(s))| - |\varphi_2(s)+\varphi_2(s-\tau_2(s)| \bigr|,
\]
so that
\[
    \bigl| |\varphi_1(s)+\varphi_1(s-\tau_1(s))| - |\varphi_2(s)+\varphi_2(s-\tau_2(s))| \leq 2 \left( 1 - \s\c^{-1} \right)^{-1} \Norm{\varphi_1 - \varphi_2}_{L^\infty(0,T)}.
\]
Thus we finally arrive at
\[
   \Norm{\Gamma\varphi_1 - \Gamma\varphi_2}_{L^\infty(0,T)} \leq
      2 T \left( 1 - \s\c^{-1} \right)^{-1}  \left( \Norm{\psi}_{L^\infty(0,\infty)} + 2(|x^0(0)| + \s S^0 + \s T) L_\psi \right)
       \Norm{\varphi_1 - \varphi_2}_{L^\infty(0,T)}
\]
and the claim follows.
\end{proof}

With Lemma \ref{lem:locEx} we constructed a unique solution of \eqref{eq:x1d}--\eqref{eq:tau1d}
on a sufficiently short time interval $[0,T]$. Obviously, we may extend the solution in time
as long as it remains Lipschitz continuous with Lipschitz constant $\s$.
But this follows directly from \eqref{cond:s},
\[
   |\dot x| \leq \psi(|x+\tilde x|)(x+\tilde x) \leq \s.
\]
Therefore, the solution is global in time.
In fact, we only need \eqref{cond:s} on the compact set $[0,2|x^0(0)|]$, since
the solution is nonincreasing for $t>0$:

\begin{lemma}
Let $x=x(t)$ be a solution of \eqref{eq:x1d}--\eqref{eq:tau1d} on the time interval $[0,T]$.
Then $|x(t)| \leq |x^0(0)|$ for all $t\in [0,T]$.
\end{lemma}

\begin{proof}
We prove that if $x\in C_\s([-S^0,T]; \R)$, then $\mbox{sign}(x+\wtx)=\mbox{sign}(x)$ for all $t\in[0,T]$.
Let us fix some $t\in[0,T]$. By Lipschitz continuity we have
\(  \label{eq:above}
   x - \s\c^{-1} |x+\wtx| = x - \s\tau \leq \wtx \leq x+\s\tau = x + \s\c^{-1} |x+\wtx|.
\)
Let us assume that $x\geq 0$ and, for contradiction, $x+\wtx<0$.
Then the left-hand side of \eqref{eq:above} gives
\[
   \wtx \geq x - \s\c^{-1} |x+\wtx| = x + \s\c^{-1} (x+\wtx),
\]
so that
\[
   \wtx \geq \frac{1+\s\c^{-1}}{1-\s\c^{-1}} x \geq 0,
\]
a contradiction to $x+\wtx<0$. We argue similarly if $x\leq 0$,
using the right-hand side of \eqref{eq:above}.
We conclude that, indeed, $\mbox{sign}(x+\wtx)=\mbox{sign}(x)$ for all $t\in[0,T]$.

Now we calculate 
\(  \label{eq:decay1d}
   \tot{}{t} x(t)^2 = - 2 \psi(|x+\wtx|) (x + \wtx) x \leq 0
\)
for all $t\in[0,T]$, which gives the claim.
\end{proof}

Note that \eqref{eq:decay1d} also implies the asymptotic consensus result
\[
   \lim_{t\to\infty} x(t) = 0.
\]
Indeed, \eqref{eq:decay1d} implies that $\lim_{t\to\infty} x(t) = \bar x$ for some $\bar x\in\R$.
Then
\[
   \lim_{t\to\infty} \tot{}{t} x(t)^2 = 
      -4 \psi(2|\bar x|) \bar x^2,
\]
and since $\psi(2|\bar x|)>0$ for $\bar x\neq 0$, we conclude that $\bar x=0$.


\section{One spatial dimension}\label{sec:1D}
We consider the system \eqref{eq:tau}--\eqref{IC:0} posed in one spatial dimension.
The local and global existence and uniqueness of solutions, given an initial datum $x^0\in C_\s((-\infty,0];\R)^N$,
is obtained analogously as in Section \ref{sec:twoAgents}.
Without loss of generality we assume that the particles at time $t=0$ are ordered according to their indices, i.e.,
\(  \label{ordering:IC:1D}
   x_1^0(0) \leq x_2^0(0) \leq \ldots \leq x_N^0(0).
\)

\begin{lemma} \label{lem:ordering1d}
The ordering \eqref{ordering:IC:1D} of the initial datum is preserved along the solution of \eqref{eq:tau}--\eqref{IC:0}, i.e.,
for all $t >0$ we have
\(   \label{ordering:1D}
   x_1(t) \leq x_2(t) \leq \ldots \leq x_N(t).
\)
\end{lemma}

\begin{proof}
The claim follows from 
the fact that whenever two particles collide,
they stick together for all future times.
Indeed, if $x_i(T) = x_j(T)$ for some $T\geq 0$
and some $i\neq j$, then $\dot x_i(T) = \dot x_j(T)$, and due to the uniqueness
of the solution, $x_i(t) = x_j(t)$ for all $t\geq T$.
\end{proof}

The following Lemma provides a general statement about trajectories
of particles traveling with speed less than $\c$.

\begin{lemma} \label{lem:nocrossing}
Fix $T\in\R$ and let the trajectories $x_i, x_j \in C_\s((-\infty, T];\R)$ with $\s<\c$.
Then we have
\(  \label{dist_est}
   \left| x_i - \wxji \right| > \frac12 \left| x_i - x_j \right| \qquad \mbox{for all } t\leq T.
\)
Moreover, if $x_i < x_j$ for some $t\leq T$, then
\( \label{ordering}
   x_i < \wxji,\qquad \wxij < x_j,\qquad \wxij < \wxji.
\)
\end{lemma}

\begin{proof}
We have
\( \label{s<c}
     \left| x_i - \wxji \right| = \c\tau > \s\tau = \left| x_j -\wxji  \right|.
\)
Then the triangle inequality gives
\[
   \left| x_j - x_i \right| \leq \left| x_j - \wxji \right| + \left| \wxji - x_i \right| < 2 \left| \wxji - x_i \right|,
\]
and \eqref{dist_est} follows.

Moreover, let us assume that $x_i < x_j$. Then, if $\wxji \leq x_i$, we would have
\[
   \left| x_i - \wxji \right| \leq \left| x_j - \wxji \right|,
\]
a contradiction to \eqref{s<c}. Therefore, $\wxji < x_i$. The other relations in \eqref{ordering}
follow similarly. See Fig. \ref{fig:fig1} for illustration.
\end{proof}

\begin{figure}
\centerline{
\includegraphics[width=0.5\columnwidth]{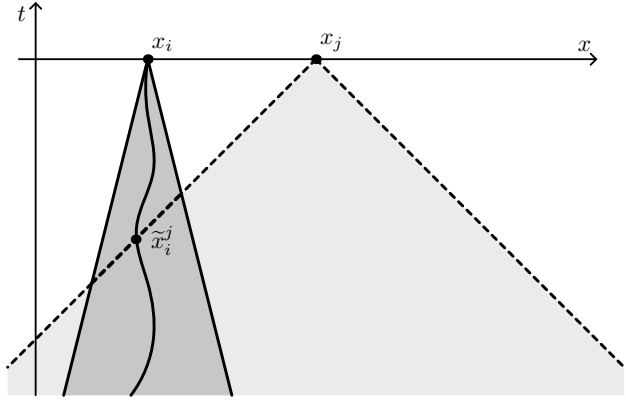}}
\caption{Agents located at $x_i$ and $x_j$ at time $t$.
The light gray area represents the light cone of agent $j$ for a given speed of light $\c>0$.
The dark gray area represents the "past cone" of agent $i$,
i.e., the set of all admissible trajectories given that the agent
travels with speed at most $\s<\c$. One possible trajectory
is plotted, and its intersection with the ray of light of agent $j$
is $\wxij$, i.e., the information about the position of $i$ that agent $j$
receives at time $t$.
}
\label{fig:fig1}
\end{figure}

\begin{theorem} \label{thm:1d}
Let $\s<\c$ and the initial datum $x^0\in C_\s([-S^0,0];\R)^N$ with $S^0=\frac{d_x(0)}{\s-\c}$.
Then we have, along the solutions of \eqref{eq:tau}--\eqref{IC:0},
\[
   \lim_{t\to\infty} d_x(t) = 0.
\]
\end{theorem}

\begin{proof}
Note that due to Lemma \ref{lem:ordering1d}, we have for the group diameter $d_x$ defined in \eqref{def:diam},
\(  \label{diam:1d}
   d_x(t) = x_N(t) - x_1(t) \qquad\mbox{for all } t\geq 0.
\)
\rev{Assumption \eqref{cond:s} implies the global Lipschitz continuity of the solution trajectories,
\[
   |\dot x_i(t)| \leq  \frac{1}{N-1} \sum_{j\neq i} \psi(|\wxji - x_i|) \left| \wxji - x_i \right| \leq \s, \qquad \mbox{for all } i=1,\ldots,N, \quad t>0,
\]
therefore, Lemma \ref{lem:nocrossing} gives}
\[
   \dot x_1 = \frac{1}{N-1} \sum_{j=1}^N \widetilde\psi_{1j} \left( \wx{j}{1} - x_1 \right) \geq 0,
\]
and similarly $\dot x_N \leq 0$.
Consequently, with \eqref{diam:1d}, $\tot{}{t} d_x \leq 0$ and there exists $d\geq 0$
such that $d_x(t) \to d$ as $t\to\infty$.
We shall prove that $d=0$.
Indeed, we have
$x_N(t) - x_1(t) \geq d$ for all $t\geq 0$.
By \eqref{dist_est} we have
\[
   d < x_N-\wx{1}{N} \leq d_x(0) \qquad\mbox{for } t>0,
\]
which in turn gives that $\psi_{N1} = \psi(|x_N-\wx{1}{N}|)$
is uniformly bounded from below by some $\underline{\psi} >0$ for all $t>T$.
Consequently, using $\dot x_1\geq 0$ and \eqref{ordering}, we have
\[
   \tot{}{t} d_x = 
      \dot x_N - \dot x_1 \leq \frac{1}{N-1} \widetilde\psi_{N1}  \left( \wx{1}{N} - x_N \right)  
      < \frac{-\underline{\psi}\, d}{N-1} \qquad \mbox{for all } t>0,
\]
which would be in contradiction to $d_x(t) \to d$ as $t\to\infty$ if $d>0$.
We conclude that $d=0$.
\end{proof}

\section{Multidimensional case}\label{sec:multiD}
Let us now consider the system \eqref{eq:tau}--\eqref{IC:0} in posed in $d$ spatial dimensions, $d\geq 2$.
The local and global existence and uniqueness of solutions, given an initial datum $x^0\in C_\s((-\infty,0];\R^d)^N$,
is obtained analogously as in Section \ref{sec:twoAgents}.
Let us point out the following principial difference to the one-dimensional case treated in Section \ref{sec:1D}.
Namely, Lemmata \ref{lem:ordering1d} and \ref{lem:nocrossing} imply that in 1D the solution trajectories always remain
inside the convex hull spanned by the initial datum at time $t=0$
(which, due to the ordering \eqref{ordering:IC:1D}, is the interval $[x_1^0(0),x_N^0(0)]$).
An analogous property does not seem to be true for the multidimensional case. In particular,
it is possible to construct initial data such that $\wxji$ for some $i$, $j$ falls outside
of the convex hull spanned by $\{x_1(0),\dots,x_N(0)\}$.
Therefore, we do not have a universal control of the convex hull spanned by the solution trajectories
as it evolves in time, and, consequently, are not able to take a geometric-like approach for proving
convergence to consensus. Instead, we are forced to follow a somehow "rougher" approach,
based on controlling the radius of the agent group, defined as
\(  \label{def:diam}
   R_x(t) := \max_{i\in\{1,\dots,N\}} |x_i(t)|.
\)
The following lemma shows that the radius is bounded uniformly in time by the radius of the initial datum,
defined as
\(  \label{def:R0}
   R_x^0 := \max_{t\in [-S^0,0]} R_x(t),\qquad \mbox{with } S^0=\frac{d_x(0)}{\s-\c}.
\)

\begin{lemma}\label{lem:Rxbound}
For $\s<\c$ let the initial datum $x^0\in C_\s([-S^0,0];\R^d)^N$ 
and let $R_x^0$ be given by \eqref{def:R0}.
Then the diameter $R_x$ defined in \eqref{def:diam} satisfies
\[
   R_x(t) \leq R_x^0 \qquad\mbox{for all } t\geq 0.
\]
\end{lemma}


\begin{proof}
Let us fix $\eps >0$. 
We shall prove that for all $t\geq 0$
\(  \label{R-eps}
   R_x(t) \leq R_x^0 + \eps.
\)
Obviously, $R_x(0) \leq R_x^0$, so that by continuity,
\eqref{R-eps} holds on the maximal interval $[0,T]$ for some $T>0$.
For contradiction, let us assume that $T<+\infty$.
Then we have
\(   \label{R-cont}
   R_x(T) = R_x^0 + \eps \qquad\mbox{and}\qquad \tot{}{t+}  R_x(T) \geq 0,
\)
where $\tot{}{t+}  R_x(T)$ denotes the right-hand side derivative of $R_x$ at $t=T$.
By continuity, there exists an index $i\in\{1,\dots,N\}$ such that $R_x(t) \equiv |x_i(t)|$
for $t\in (T,T+\delta)$ for some $\delta>0$.
Then we calculate
\[
   \tot{}{t+} R_x(T) = \tot{}{t+} |x_i(T)|^2 &=& \frac{2}{N-1} \sum_{j=1}^N \widetilde\psi_{ij} \left[x_j(T-\tau_{ij}(T))- x_i(T)\right]\cdot x_i(T) \\
     &=& \frac{2}{N-1} \sum_{j=1}^N  \widetilde\psi_{ij} \left[ x_j(T-\tau_{ij}(T))\cdot x_i(T) - |x_i(T)|^2\right],
\]
where $\tot{}{t+}$ denotes the derivative with respect to $t$ from the right-hand side.
By definition, we have for all $j\in\{1,\dots,N\}$,
\[
   |x_j(T-\tau_{ij}(T))| \leq R_x^0 + \eps = |x_i(T)|,
\]
so that the Cauchy-Schwarz inequality yields $\tot{}{t+} R_x(T) \leq 0$.
Moreover, we can have $\tot{}{t+} R_x(T) = 0$ only if equality takes place
in the Cauchy-Schwarz inequality, i.e., if
\[
   x_j(T-\tau_{ij}(T))\cdot x_i(T) = |x_i(T)|^2 \qquad\mbox{for all } j\in\{1,\dots,N\}.
\]
That would mean that $x_j(T-\tau_{ij}(T)) = x_i(T)$ for all $j$, which in turn gives $\tau_{ij}(T)=0$ and $x_j(T) = x_i(T)$ for all $j$,
i.e., the system reached equilibrium at time $T$ and does not evolve further.
Otherwise, we have $\tot{}{t+} R_x(T) < 0$, which is a contradiction to \eqref{R-cont}.
Consequently, \eqref{R-eps} is indeed valid for all $t\geq 0$,
and we conclude by taking the limit $\eps\to 0$.
\end{proof}

The next lemma provides a control of the diameter $d_x(t)$ of the solution.

\begin{lemma}\label{lem:shrinkage}
For $\s<\c$ let the initial datum $x^0\in C_\s([-S^0,0];\R^d)^N$ 
and let $R_x^0$ be given by \eqref{def:R0}.
Then we have
\[
    \tot{}{t} d_x \leq \left[  \frac{2\s}{\c-\s}\overline{\psi} - \frac{N}{N-1}\underline{\psi} \right] d_x \qquad\mbox{for almost all } t\in(0,\infty),
\]
where
\(   \label{psipsi}
    \underline\psi := \min_{r\in [0,R_x^0]} \psi(r), \qquad
    \overline\psi := \max_{r\in [0,R_x^0]} \psi(r).
\)
\end{lemma}

\begin{proof}
Due to the continuity of the solution trajectories $x_i(t)$,
there is an at most countable system of open, mutually disjoint
intervals $\{\mathcal{I}_\sigma\}_{\sigma\in\N}$ such that
\[
   \bigcup_{\sigma\in\N} \overline{\mathcal{I}_\sigma} = [0,\infty)
\]
and for each ${\sigma\in\N}$ there exist indices $i(\sigma)$, $k(\sigma)$
such that
\[
   d_x(t) = |x_{i(\sigma)}(t) - x_{k(\sigma)}(t)| \quad\mbox{for } t\in \mathcal{I}_\sigma.
\]
Then, using the abbreviated notation $i:=i(\sigma)$, $k:=k(\sigma)$,
we have for every $t\in \mathcal{I}_\sigma$,
\(  \label{shrinkage:1}
   \frac12 \tot{}{t} d_x(t)^2 &=& \frac12 \tot{}{t} |x_i - x_k|^2 \\
    &=& \frac{1}{N-1} \sum_{j=1}^N \widetilde\psi_{ij} \bigl(\wxji - x_i\bigr)\cdot (x_i-x_k)
      - \frac{1}{N-1} \sum_{j=1}^N \widetilde\psi_{kj} \bigl(\wx{j}{k} - x_k\bigr)\cdot (x_i-x_k).   \nonumber
\)
Let us work on the first term of the right-hand side. We have for any $j\in\{1,\dots,N\}$,
\(   \label{shrinkage:2}
   \widetilde\psi_{ij} \left(\wxji - x_i\right)\cdot (x_i-x_k) = \widetilde\psi_{ij} \left(\wxji - x_j\right)\cdot (x_i-x_k) + \widetilde\psi_{ij} \left(x_j- x_i\right)\cdot (x_i-x_k).
\)
By \eqref{eq:tau} we have
\[
   |\wxji - x_i| = \c \tau_{ij}.
\]
On the other hand, since $x_j$ can travel with speed at most $\s$,
\[
   |\wxji - x_j| \leq \s \tau_{ij}.
\]
Therefore, by the triangle inequality,
\[
   \c \tau_{ij} = |\wxji - x_i| \leq |x_i-x_j| + |\wxji-x_j| \leq d_x + \s\tau_{ij}, 
\]
so that
\[
    |\wxji-x_j| \leq \frac{\s}{\c-\s} d_x.
\]
Consequently, with the bound $\psi\leq \overline{\psi}$ given by \eqref{psipsi} and the Cauchy-Schwarz inequality,
we estimate the first term of the right-hand side of \eqref{shrinkage:2} by
\[
   \widetilde\psi_{ij} \left(\wxji - x_j\right)\cdot (x_i-x_k) \leq \overline{\psi} |\wxji-x_j| |x_i-x_k| \leq \frac{\s\overline{\psi}}{\c-\s} d_x^2.
\]
For the second term in \eqref{shrinkage:2}, we observe, using the Cauchy-Schwarz inequality,
\[
   (x_j- x_i)\cdot (x_i-x_k) &=& (x_j-x_k)\cdot(x_i-x_k) - |x_i-x_k|^2 \\
      &\leq& |x_i-x_k| \bigl( |x_j-x_k| - |x_i-x_k| \bigr) \leq 0,
\]
since, by definition, $|x_j-x_k| \leq d_x = |x_i-x_k|$.
Moreover, with Lemma \ref{lem:Rxbound} we have
\[
   \widetilde \psi_{ij} \geq \min_{r\in [0,R_x^0]} \psi(r) = \underline{\psi} > 0.
\]
Consequently,
\[
    \widetilde \psi_{ij} \left(x_j- x_i\right)\cdot (x_i-x_k) \leq \underline{\psi} \, (x_j- x_i)\cdot (x_i-x_k).
\]
Carrying out analogous steps for the second term of the right-hand side of \eqref{shrinkage:1},
we finally obtain
\[
    \frac12 \tot{}{t} d_x^2 &\leq& \frac{2\s\overline{\psi}}{\c-\s} d_x^2 + \frac{\underline{\psi}}{N-1} \sum_{j=1}^N 
      \bigl[(x_j- x_i) - (x_j-x_k)\bigr]\cdot (x_i-x_k) \\
      &=&  \left[ \frac{2\s}{\c-\s}\overline{\psi} - \frac{N}{N-1} \underline{\psi} \right] d_x^2.
\]
This immediately gives the statement.
\end{proof}

Direct consequence of Lemma \ref{lem:shrinkage} is the following result about asymptotic consensus
for the system \eqref{eq:tau}--\eqref{IC:0} in the multidimensional setting.

\begin{theorem}\label{thm:shrinkage}
For $\s<\c$ let the initial datum $x^0\in C_\s([-S^0,0];\R^d)^N$, 
let $R_x^0$ be given by \eqref{def:R0} and let $\underline\psi$, $\overline\psi$ be given by \eqref{psipsi}.
If the condition
\(   \label{multiDcond}
   \frac{2\s}{\c-\s} < \frac{N\underline{\psi}}{(N-1)\overline{\psi}}
\)
is verified, then all solutions of \eqref{eq:tau}--\eqref{IC:0} reach asymptotic consensus,
i.e., $d_x(t) \to 0$, exponentially fast as $t\to\infty$.
\end{theorem}

\rev{
Condition \eqref{multiDcond} can be interpreted, for a fixed influence function $\psi$, as a smallness condition
on the speed limit $\s$ with respect to the speed of light $\c$. For instance, if we choose $\psi$ constant on $[0,R_x^0]$,
then $\underline{\psi} = \overline{\psi}$ and \eqref{multiDcond} it is satisfied for all $N\in\N$ if $3\s \leq \c$. 
Alternatively, for fixed $\s$ and $\c$ it can be interpreted
as a condition on slow enough decay of $\psi$. We admit that \eqref{multiDcond} is relatively restrictive,
reducing the strength of claim of Theorem \ref{thm:shrinkage}.
In particular, we hypothesize that also in the multidimensional setting
all solutions of \eqref{eq:tau}--\eqref{IC:0} reach asymptotic consensus
as long as $\psi(s)>0$ for all $s\geq 0$, regardless of the particular values of
$\s$, $\c$, $\underline{\psi}$ and $\overline{\psi}$.
However, the presence of the state-dependent and heterogeneous delays $\tau_{ij}$
prohibits the application of all techniques invented so far for study of consensus systems with delay
that are known to us. Consequently, a proof of a (hypothetical) optimal consensus result in multiple spatial dimensions
seems to require development of new methods and will be subject of a future work.
}

\section*{Acknowledgment}
JH acknowledges the support of the KAUST baseline funds.


\end{document}